\documentclass[a4paper,UKenglish,numberwithinsect,compactauthor]{lipics-v2021}
\hideLIPIcs
\usepackage{microtype}
\usepackage{xcolor}
\usepackage{mathtools}
\usepackage{placeins}
\usepackage{tikz}
\usetikzlibrary{arrows.meta,positioning,calc}

\IfFileExists{orcid.pdf}{\renewcommand{\orcidsymbol}{\includegraphics[height=9pt]{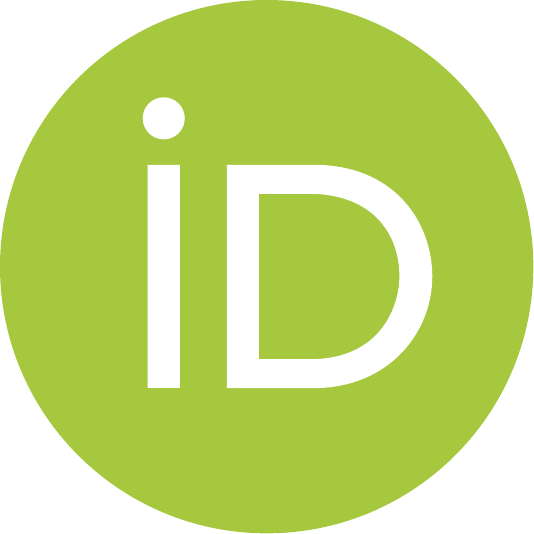}}}{}
\IfFileExists{fontawesome5.sty}{}{%
  \renewcommand{\mailsymbol}{%
    \tikz[baseline=0pt,x=1pt,y=1pt]{%
      \draw[lipicsGray,line width=.7pt,line join=round]
        (0,0) rectangle (10,7) (0,7) -- (5,3.2) -- (10,7);}}}

\title{An Asymptotically Tight Bound on the Number of Distinct Squares in Circular Words}
\titlerunning{Distinct Squares in Circular Words}
\author{Rikuya Hamai}
  {Department of Information Science and Technology, Kyushu University, Fukuoka, Japan}
  {hamai.rikuya.226@s.kyushu-u.ac.jp}
  {https://orcid.org/0009-0006-9575-3911}
  {Supported by JST BOOST Grant Number JPMJBS2406.}
\authorrunning{R. Hamai}
\Copyright{Rikuya Hamai}

\newcommand{\word}[1]{\texttt{#1}}
\newcommand{\Sq}{\mathrm{SQ}}
\newcommand{\Fac}{\mathit{Fac}}

\newcommand{\Zset}{\mathcal R}
\newcommand{\outdeg}{\operatorname{outdeg}}

\tikzset{
 vertex/.style={draw,rounded corners=1.5pt,fill=white,minimum width=13mm,
                minimum height=8mm,font=\ttfamily\normalsize},
 edgea/.style={-{Stealth[length=2.2mm]},line width=0.8pt},
 edgeb/.style={-{Stealth[length=2.2mm]},line width=0.8pt,densely dashed},
 elabel/.style={font=\ttfamily\small,fill=white,inner sep=2pt}}

\begin{document}
\nolinenumbers
\maketitle

\begin{abstract}
A square is a word $xx$ with $x$ nonempty.
We study the maximum number of distinct squares occurring
in the rotations of a word of length $n$.
Amit and Gawrychowski \ (SPIRE 2017) obtained an upper bound of $3.14n$.
Charalampopoulos et al.\ (CPM 2026) improved the upper bound
to $\lceil9n/5\rceil$ and gave a binary construction with
at least $3n/2-O(\sqrt n)$ distinct squares for infinitely many $n$.
They conjectured an upper bound of $3n/2$.
Li and Song subsequently reported an improved upper bound
of $5n/3$.
We prove that a circular word of length $n$ containing $\sigma$
distinct letters has at most $\lfloor3(n-\sigma)/2\rfloor$
distinct squares.
This settles the conjecture and establishes that the leading
coefficient $3/2$ is optimal even for binary words.
Our proof groups squares by their Lyndon roots and, for primitive
words, uses a weighted count of circuits in the Rauzy graphs
of the bi-infinite repetition of the word.
\end{abstract}
% Separate title page, excluded from the CPM 2027 body-page limit.
\clearpage

\section{Introduction}

A \emph{square} is a word $xx$, where $x$ is nonempty.
How many distinct squares can a word of length $n$ contain?
In 1998, Fraenkel and Simpson~\cite{fraenkel-simpson} proved that
the number is less than $2n$.
Their proof counts each square at its rightmost occurrence and
shows that at most two squares are counted at any position.
They also constructed words with at least $n-O(\sqrt n)$ distinct
squares for infinitely many lengths and conjectured that every
nonempty word contains fewer distinct squares than its length.
This conjecture led to a series of improvements to the upper
bound~\cite{ilie,deza-franek-thierry,thierry}.

Deza, Franek, and Jiang~\cite{deza-franek-jiang} proposed a
bound that takes the alphabet size into account.
They conjectured that a word of length $n$ containing $\sigma$
distinct letters has at most $n-\sigma$ distinct squares whose
halves are primitive, that is, are not proper powers.
Brlek and Li~\cite{brlek-li} proved the $n-\sigma$ bound for all
distinct squares, resolving both conjectures.
Their proof uses independent circuits in Rauzy graphs~\cite{rauzy},
whose vertices and arcs correspond to factors of two consecutive
lengths.
Brlek and Li also obtained an alphabet-independent upper bound
of $n-\Theta(\log n)$~\cite{brlek-li-log}.
Tomita and I~\cite{tomita-i} gave another proof of this latter bound
by grouping circuits according to their Lyndon roots.
Together with the construction of Fraenkel and Simpson, these
results establish that the optimal leading coefficient for
linear words is one.

\paragraph*{Circular words.}
We study the corresponding problem for circular words.
For a word $w$ of length $n$, the circular word $[w]$ is the set
of its rotations.
We count a square once if it is a factor of at least one of these
rotations, so only squares of length at most $n$ are counted.
For example, the linear word $\word{abcabc}$ has only one square
of length six, whereas the circular word also contains
$\word{bcabca}$ and $\word{cabcab}$.
Thus closing a word into a circle can create additional distinct
squares.
Write $S([w])$ for the number of distinct squares in $[w]$, and let
$\mathrm{CS}(n)$ be the maximum of $S([w])$ over all words of length
$n\ge1$, with no restriction on the alphabet size.

Amit and Gawrychowski~\cite{amit-gawrychowski} first studied this
maximum and proved an upper bound of $3.14n$.
Their proof used the upper bound then available for linear words.
They also constructed circular words with at least
$5n/4-O(1)$ distinct squares for infinitely many lengths.
This construction shows that the optimal leading coefficient
for circular words is strictly larger than that for linear words.

The result of Brlek and Li gives a smaller upper bound directly:
every square in $[w]$ occurs in $ww$, so
$S([w])\le 2n-\sigma$.
Charalampopoulos et al.~\cite{circular} improved the upper bound
to $\lceil9n/5\rceil$ by analysing the starting positions of
long squares.
They also gave a binary construction based on Fibonacci words
with at least $3n/2-O(\sqrt n)$ distinct squares for infinitely
many lengths.
They conjectured that $\mathrm{CS}(n)\le3n/2$ for every $n$ and
suggested adapting the graph-theoretic proof of Brlek and Li.
Subsequently, Li and Song~\cite{li-song} reported an improved
upper bound of $5n/3$ using Rauzy graphs.

\paragraph*{Our result.}
We prove the $3n/2$ conjecture in a form that also accounts for
the alphabet size.

\begin{theorem}\label{thm:main}
Let $w$ be a nonempty word of length $n$ containing $\sigma$
distinct letters. Then
\[
 S([w])\le
 \left\lfloor\frac{3(n-\sigma)}{2}\right\rfloor.
\]
\end{theorem}

Together with the binary construction of Charalampopoulos
et al.~\cite{circular}, Theorem~\ref{thm:main} gives
\[
 \limsup_{n\to\infty}\frac{\mathrm{CS}(n)}{n}=\frac32.
\]
Thus the optimal leading coefficient for circular words is $3/2$.
Charalampopoulos et al.~\cite[Lemma~14]{circular} had already
proved the $3n/2$ bound for nonprimitive words.
For these words, we obtain the stronger bound
$S([w])\le3n/2-2\sigma$.
The main new case is therefore that of primitive words.

\paragraph*{Proof overview.}
As in the work of Brlek and Li~\cite{brlek-li} and Tomita and I~\cite{tomita-i}, we charge squares to circuits in Rauzy graphs, here those of the bi-infinite periodic word $W=w^{\mathbb Z}$.
Let $m$ be the length of the primitive root of $w$.
For the squares $xx$ with $|x|<m$, each Lyndon root
receives at least as many circuits as it has squares,
and the total number of these circuits is at most $2(m-\sigma)$.
This is enough for nonprimitive words, whose longer squares are easy to count exactly (Section~\ref{sec:nonprimitive}).
For primitive words, the only problematic case is a graph consisting of two cycles through a common vertex, each carrying a charged circuit.
Since $W$ has no period shorter than~$n$, reading $W$ must go around one of these cycles twice in a row, and this cycle then also appears in the graphs of higher orders (Lemma~\ref{lem:extension}).
Its root thus has circuits to spare, so we may count this cycle only half, which brings the total down to $3(n-\sigma)/2$ (Lemma~\ref{lem:reduction}).

\section{Preliminaries}
\subsection{Words and circular factors}

Fix a finite alphabet $\Sigma$ with a total order. In the examples
we order the letters alphabetically.
A word over $\Sigma$ is a finite sequence of letters from $\Sigma$.
Write $|u|$ for the length of $u$ and $\varepsilon$ for the empty
word. Positions in finite words are numbered from~1. For integers
$i,j$, write $u[i]$ for the letter at position $i$ and $u[i..j]$
for the factor from position $i$ through position $j$, inclusive,
with $u[i..j]=\varepsilon$ if $i>j$.
A factor of the form $u[1..j]$ is a \emph{prefix}.

We write $xy$ for the concatenation of $x$ and $y$, and $x^k$
for the concatenation of $k$ copies of $x$, where $k$ is a nonnegative
integer and $x^0=\varepsilon$. A \emph{square} is a
word $xx$ with $x\ne\varepsilon$.

Two words $u$ and $v$ are \emph{conjugate} if $u = xy$ and $v = yx$ for some words $x$ and~$y$.
For a nonempty word $w$ of length~$n$, the \emph{rotations} of $w$ are the words $w[i..n]w[1..i-1]$ with $1 \le i \le n$.
They form the conjugacy class
\[
  [w] = \{\, w[i..n]w[1..i-1] : 1 \le i \le n \,\},
\]
which we call the \emph{circular word} represented by $w$.
Intuitively, $[w]$ is $w$ written around a circle, and every rotation of $w$ represents the same circular word.
A word \emph{occurs} in $[w]$ if it is a factor of some rotation of $w$.
We then also call it a \emph{circular factor} of $w$; its length is at most $n$.
From now on, $w$ is a fixed nonempty word of length~$n$, and $\sigma$ is the number of distinct letters in~$w$. Write $W=w^{\mathbb Z}$ for the bi-infinite
periodic word with $W[1..n]=w$ and $W[i+n]=W[i]$ for every
$i\in\mathbb Z$. For $0\le j\le n$, the length-$j$ circular
factors of $w$ are precisely the length-$j$ factors of $W$, or
equivalently of $ww$. 

For a finite or bi-infinite word $Y$ and an integer $j\ge0$, let $\Fac_Y(j)$ denote its set
of length-$j$ factors, with $\Fac_Y(0)=\{\varepsilon\}$, and let
$\Sq(Y)$ denote its set of square factors. The square factors of
$[w]$ and their number are therefore given by
\begin{equation}\label{eq:squares}
 \Sq([w])=\{xx\in\Sq(W):2|x|\le n\},\qquad
 S([w])=|\Sq([w])|.
\end{equation}

An integer $p$ with $1\le p\le |x|$ is a \emph{period} of a
nonempty finite word $x$ if $x[i]=x[i+p]$ for
$1\le i\le |x|-p$. A positive integer $p$ is a period of a
bi-infinite word if this equality holds at every position.
A nonempty word is \emph{primitive} if it is not a power $x^k$
with $k\ge2$, and \emph{nonprimitive} otherwise. Every nonempty word
is a power of a unique primitive word, its \emph{primitive root}~\cite{lothaire}.

Write $w=u^k$, where $u$ is primitive and $k\ge1$, and put $m=|u|$. Then
$n=km$, $W=u^{\mathbb Z}$, and $u$ contains the same $\sigma$
letters as $w$. The least period of $W$ is $m$. Indeed, if $h$
is a period of $W$, the letters of $u$ are constant on residue
classes modulo $\gcd(h,m)$. Since $u$ is primitive, $\gcd(h,m)=m$.
In particular, the $m$ rotations of $u$ are distinct.

\subsection{Rauzy graphs}\label{sec:rauzy}

For a directed graph $G$, write $V(G)$ and $E(G)$ for its vertex
and arc sets. We allow loops and parallel arcs. A \emph{walk} is a
sequence of consecutive arcs, with repetitions allowed.
The length of a finite walk is
the number of arcs it traverses, counted with multiplicity. A \emph{circuit} is a nonempty
finite closed walk. It is \emph{elementary} if its vertices are distinct
except for the final return to the first vertex. We identify circuits
whose sequences of arcs agree up to cyclic rotation.
A graph is \emph{strongly connected} if each vertex can be reached
from every other vertex by a walk. We write $\outdeg(v)$ for the
number of arcs leaving $v$.

\begin{definition}[Rauzy graph]\label{def:graph}
For a finite or bi-infinite word $Y$ and an integer $\ell\ge0$, the
\emph{Rauzy graph of order $\ell$}, denoted by $\Gamma_Y(\ell)$, has
vertex set $\Fac_Y(\ell)$ and one arc for each factor $q\in\Fac_Y(\ell+1)$,
\[
 q[1..\ell]\xrightarrow{\ q\ }q[2..\ell+1].
\]
\end{definition}
The \emph{walk read from $W$} at order~$\ell$ is the bi-infinite walk whose vertex at step $i \in \mathbb{Z}$ is $W[i..i+\ell-1]$ and whose arc from step~$i$ to step~$i+1$ is labelled $W[i..i+\ell]$. For $r\ge0$, its $r$ steps starting at $i$ read the factor $W[i..i+\ell+r-1]$.
However, a walk in the graph need not arise in this way. For example, the loop
labelled $\word{aa}$ in $\Gamma_W(1)$ for
$W=(\word{aab})^{\mathbb Z}$ can be traversed twice in the graph,
although $\word{aaa}$ does not occur in $W$.

The \emph{factor complexity} of $W$ at length $\ell$ is the number
$P_\ell$ of its distinct length-$\ell$ factors. Let $d_\ell$
denote the first difference of this function. For every $\ell\ge0$,
\begin{equation}\label{eq:d}
 P_\ell=|\Fac_W(\ell)|,\qquad
 d_\ell=P_{\ell+1}-P_\ell
       =|E(\Gamma_W(\ell))|-|V(\Gamma_W(\ell))|.
\end{equation}

\begin{lemma}[The walk induced by one period]\label{lem:connect}
For each $\ell \ge 0$, the vertices $W[i..i+\ell-1]$ for $1 \le i \le m+1$,
joined in order by the arcs labelled $W[i..i+\ell]$ for $1 \le i \le m$,
form a closed walk of length~$m$ in $\Gamma_W(\ell)$.
This walk visits every vertex and traverses every arc of $\Gamma_W(\ell)$.
\end{lemma}

\begin{proof}
The first and last vertices coincide because $W=u^{\mathbb{Z}}$ has period~$m$.
Every factor of $W$ of length $\ell$ or $\ell+1$ has an occurrence starting
at one of the positions $1,\dots,m$, so the walk visits every vertex and
traverses every arc.
\end{proof}

Thus $\Gamma_W(\ell)$ is strongly connected, has at most $m$ vertices,
and has positive outdegree at every vertex. Summing outdegrees gives
\begin{equation}\label{eq:outdegrees}
 d_\ell=\sum_{v\in V(\Gamma_W(\ell))}(\outdeg(v)-1)\ge0 .
\end{equation}

\begin{lemma}[Stabilisation of factor complexity]\label{lem:telescoping}
For every $\ell\ge m-1$, $P_\ell=m$.
\end{lemma}

\begin{proof}
The length-$m$ factors starting at positions $1,\ldots,m$ are the
distinct rotations of $u$. All have the same number of occurrences
of each letter, so their length-$(m-1)$ prefixes determine their
last letters. Thus, these $m$ prefixes are pairwise distinct, so $P_{m-1}\ge m$.
Since $P_\ell\le m$ for every $\ell$, we have $P_{m-1}=m$.
By~\eqref{eq:outdegrees}, $P_\ell$ is nondecreasing in $\ell$,
so $P_\ell=m$ for all $\ell\ge m-1$.
\end{proof}

Consequently $d_\ell=0$ for $\ell\ge m-1$. Since $P_1=\sigma$,
for $m\ge2$ we obtain
\begin{equation}\label{eq:telescoping}
 \sum_{\ell=1}^{m-2}d_\ell=P_{m-1}-P_1=m-\sigma .
\end{equation}

\begin{lemma}[Orders with $d_\ell=0$]\label{lem:cycle}
If $d_\ell=0$, then $\Gamma_W(\ell)$ consists of a single
elementary circuit of length $m$.
\end{lemma}

\begin{proof}
By~\eqref{eq:outdegrees}, every vertex has outdegree one.
Strong connectivity (Lemma~\ref{lem:connect}) then makes the graph a single elementary circuit.
Let $h$ be its length, which equals the number $P_\ell\le m$ of vertices.
Since every vertex has only one outgoing arc, the walk read from $W$ goes around this circuit again and again, so its arcs at steps $i$ and $i+h$ coincide for every $i\in\mathbb Z$.
The arc at step $i$ is labelled $W[i..i+\ell]$, whose first letter is $W[i]$.
Hence $W[i+h]=W[i]$ for every $i$, so $h$ is a period of $W$.
Since the least period of $W$ is $m$, we have $h\ge m$, and thus $h=m$.
\end{proof}
For a circuit $C$ in a directed graph $G$, let
$\mu(C)\in\mathbb R^{E(G)}$ have coordinate $\mu(C)_e$ equal to the
number of times $C$ traverses arc $e$. We call circuits
\emph{independent} if their vectors are linearly independent.

In a finite strongly connected directed graph $G$, any set of independent
circuits has size at most $|E(G)|-|V(G)|+1$~\cite[Chap.~IV, Thm.~2]{berge} (see also Theorem~7 of~\cite{tomita-i}).
By Lemma~\ref{lem:connect}, $\Gamma_W(\ell)$ is strongly connected.
Thus, by~\eqref{eq:d}, it has at most $d_\ell+1$ independent circuits.

\section{Counting squares by their roots}\label{sec:circuits}

\subsection{A partition of the square factors}
Recall that $w=u^k$, where $|w|=n$, $u$ is primitive, and $m=|u|$.
This subsection reduces Theorem~\ref{thm:main} to bounding the number of short squares.
We count the squares $xx$ in $[w]$ with $|x|\ge m$ exactly (Lemma~\ref{lem:long-powers}) and group the remaining squares by the conjugacy class of their primitive root, which gives the decomposition~\eqref{eq:root-data}.

\begin{lemma}[Squares of length at least $2m$]\label{lem:long-powers}
The squares $xx$ in $[w]$ with $|x|\ge m$ are precisely
\[
  v^{2j},\qquad v\in[u],\quad 1\le j\le\lfloor k/2\rfloor .
\]
There are $m\lfloor k/2\rfloor$ such squares.
\end{lemma}

\begin{proof}
Consider an occurrence of $xx$ in $W$ with $m\le |x|\le n/2$.
The length-$m$ factors at the beginnings of the two copies of $x$
coincide. Since the $m$ rotations of $u$ are all different, the distance between the two starting positions is a multiple of $m$.
This distance is $|x|$, so $|x|=jm$ for some
$1\le j\le\lfloor k/2\rfloor$.
Let $v$ be the length-$m$ prefix of $x$. Then $v\in[u]$,
and $x=v^j$.
Conversely, every word $v^{2j}$ with $v\in[u]$ and
$1\le j\le\lfloor k/2\rfloor$ is a factor of $W$ of length
at most $n$, and hence occurs in $[w]$.
These squares are pairwise distinct, since the length determines $j$
and the length-$m$ prefix determines $v$.
\end{proof}

A \emph{Lyndon word} is a primitive word that is lexicographically
smallest among its rotations. Every square has a unique expression
$a^{2r}$ with $a$ primitive and $r\ge1$. We call the Lyndon word in
$[a]$ its \emph{Lyndon root}. For example, the squares
$\word{aabaab}$, $\word{abaaba}$, and $\word{baabaa}$ all have
Lyndon root $\word{aab}$.

For a Lyndon word $z$, let $\Sq([w],z)$ be the set of squares in
$\Sq([w])$ whose Lyndon root is $z$.
Write $\Zset$ for the set of Lyndon words $z$ of length less than
$m$ for which $\Sq([w],z)\ne\varnothing$. For each $z\in\Zset$,
put $p_z=|z|$ and $s_z=|\Sq([w],z)|$.
Consider a square $xx$ in $[w]$. If $|x|<m$, its primitive root
has length at most $|x|$, hence less than $m$.
If $|x|\ge m$, Lemma~\ref{lem:long-powers} shows that its primitive
root is a rotation of $u$ and hence has length $m$. This partitions
the squares in $[w]$ and gives
\begin{equation}\label{eq:root-data}
  S([w])=m\lfloor k/2\rfloor+\sum_{z\in\Zset}s_z .
\end{equation}
When $m=1$, we have $\sigma=1$ and $S([w])=\lfloor n/2\rfloor$.
For $n=1$ both sides of Theorem~\ref{thm:main} are zero. For $n\ge2$,
the result follows from $n/2\le 3(n-1)/2$ by taking floors.
Henceforth assume $m\ge2$. 

To bound the sum in~\eqref{eq:root-data},
we associate circuits in the Rauzy graphs of $W$ with each root
$z\in\Zset$. All subsequent sums over $z$ range over $\Zset$.

\subsection{Circuits associated with a Lyndon root}
We show that the orders counted for each root form an interval
(Lemma~\ref{lem:order-interval}) and use its length to bound the
number of squares with that root (Lemma~\ref{lem:root-count}).

Fix a Lyndon word $z$ of length $p$. For each $\ell\ge p-1$,
reading one period of $z^{\mathbb Z}$ in $\Gamma_{z^{\mathbb Z}}(\ell)$
gives a circuit $C_z(\ell)$ whose successive arc labels are
\[
 z^{\mathbb Z}[1..\ell+1],\quad\ldots,\quad z^{\mathbb Z}[p..\ell+p].
\]
This circuit has length $p$. For $\ell\ge p$ it is elementary, since the length-$p$ prefixes of its successive vertices are the $p$ rotations of $z$, which are distinct because $z$ is primitive. For any finite or bi-infinite word
$Y$, the graph $\Gamma_Y(\ell)$ contains $C_z(\ell)$ if and
only if $\Fac_{z^{\mathbb Z}}(\ell+1)\subseteq\Fac_Y(\ell+1)$.
We call $z$ the Lyndon root of this circuit. Choosing a
different starting vertex may change the rotation read on the circuit,
but its Lyndon root is the same.

\begin{lemma}[Orders counted for a fixed root]\label{lem:order-interval}
For every $z\in\Zset$, the orders $\ell\ge p_z$ at which
$\Gamma_W(\ell)$ contains $C_z(\ell)$ form a nonempty integer
interval $\{p_z,\ldots,M_z\}$, where
\begin{equation}\label{eq:counts}
 M_z=\max\{\ell\ge p_z:\Fac_{z^{\mathbb Z}}(\ell+1)\subseteq\Fac_W(\ell+1)\}
 \le m-2 .
\end{equation}
\end{lemma}

\begin{proof}
Take $a^{2r}\in\Sq([w],z)$ with $a\in[z]$. The length-$(p_z+1)$
factors of its prefix $aa$ starting at positions $1,\ldots,p_z$
give all the arc labels of $C_z(p_z)$. If every arc label of
$C_z(L)$ occurs in $W$, their length-$(\ell+1)$ prefixes give
every arc label of $C_z(\ell)$ for $p_z\le\ell\le L$
(cf.\ Lemmas~10 and~11 of~\cite{tomita-i}).
For $\ell\ge m-1$, Lemmas~\ref{lem:telescoping} and~\ref{lem:cycle}
show that $\Gamma_W(\ell)$ is a single elementary circuit of length
$m$. It cannot contain $C_z(\ell)$, an elementary circuit of
length $p_z<m$. This proves both the interval assertion and the bound
on its endpoint.
\end{proof}

Let $K_z=M_z-p_z+1$ be the number of orders in this interval.
At such an order we say that $z$ is \emph{counted} and call
$C_z(\ell)$ a \emph{counted circuit}.

In particular, no root is counted at orders $\ell\ge m-1$.
There, $\Gamma_W(\ell)$ is a single elementary circuit of length~$m$ that reads the rotations of~$u$, and the squares with these primitive roots are exactly those counted in Lemma~\ref{lem:long-powers}.
Moreover, $d_\ell=0$ at these orders, so they play no role in~\eqref{eq:telescoping}.

\begin{lemma}[Number of circuits for a fixed root]\label{lem:root-count}
For every $z\in\Zset$,
\begin{equation}\label{eq:class-bound}
 K_z\ge s_z,\qquad K_z\ge 2s_z-p_z .
\end{equation}
\end{lemma}

\begin{proof}
We apply the rotation argument of Lemma~16 of~\cite{tomita-i}
directly to the squares in $[w]$. Put $p=|z|$. Since $s_z>0$, there
is a largest integer $r\ge1$ for which $a^{2r}\in\Sq([w],z)$ for
some rotation $a$ of $z$. Let $t$ be the number of rotations that
attain this maximum. For $1\le j<r$ and every rotation $b$ of
$z$, the square $b^{2j}$ occurs in any one of these longest squares.
Indeed, its starting offset can be taken in $\{0,\ldots,p-1\}$, and
$(p-1)+2jp\le2rp$. Every square with root $z$ has the form
$b^{2j}$ with $1\le j\le r$, and its length and first $p$ letters
determine $j$ and $b$. Hence
\[
 s_z=p(r-1)+t .
\]

Index the $p$ rotations by their starting positions modulo $p$
in $z^{\mathbb Z}$, and mark the $t$ positions attaining $r$.
Set $L=2pr-(p-t)$. For any position $i$ modulo $p$, at least one
of $i-(p-t),\ldots,i$ is marked, since there are only $p-t$
unmarked positions. If the chosen position is $d$ steps before $i$,
the corresponding square of length $2pr$ contains the length-$L$
factor starting at residue $i$, because $d+L\le2pr$.
Since the corresponding square occurs in $W$, this length-$L$
factor also occurs in $W$. As this holds for every residue $i$,
we obtain $\Fac_{z^{\mathbb Z}}(L)\subseteq\Fac_W(L)$. 
Thus, equation~\eqref{eq:counts}
gives $M_z\ge L-1$ and hence
\[
 K_z\ge L-p=2p(r-1)+t.
\]
The last expression is both $s_z+p(r-1)$ and $2s_z-t$.
The two inequalities follow from $r\ge1$ and $t\le p$.
\end{proof}

Summing the first inequality in~\eqref{eq:class-bound} over the roots
bounds the number of squares of length less than $2m$ by $\sum_z K_z$. We will use the second
inequality when reducing circuit weights in Section~\ref{sec:weighted}.

\subsection{Counting circuits by graph order}

For each order $1\le\ell\le m-2$, let
$b_\ell=|\{z\in\Zset:p_z\le\ell\le M_z\}|$ be the number of roots
counted at that order. At order $\ell=p_z-1$, the $p_z$ successive vertices of $C_z(\ell)$ in one period are distinct by Lemma~\ref{lem:telescoping} applied to $z$, so the circuit is elementary; we nevertheless count $z$ only from order~$p_z$ on, which suffices for Lemma~\ref{lem:root-count}.
Since $M_z\le m-2$ by~\eqref{eq:counts}, counting the pairs $(z,\ell)$ first by root and then by order gives
\begin{equation}\label{eq:double-count}
 \sum_z K_z=\sum_{\ell=1}^{m-2} b_\ell .
\end{equation}
We use the least-arc argument from the proof of Lemma~14
of~\cite{tomita-i} for the circuits at a single order.

\begin{lemma}[Independence at a fixed order]\label{lem:independence}
Let $Y$ be a finite or bi-infinite word and $\ell\ge0$ an integer.
For each Lyndon word $z$ with $|z|\le\ell+1$ and
$\Fac_{z^{\mathbb Z}}(\ell+1)\subseteq\Fac_Y(\ell+1)$, consider the circuit
$C_z(\ell)$. These circuits are pairwise distinct and linearly
independent in $\Gamma_Y(\ell)$.
\end{lemma}

\begin{proof}
Write $e_z$ for the least arc label of $C_z(\ell)$ in lexicographic
order. Each arc label has length at least $|z|$, and its first
$|z|$ letters specify a rotation of $z$. Hence $e_z$ is the
length-$(\ell+1)$ prefix of the infinite repetition starting with $z$.

These least labels are distinct. Let $z_1\ne z_2$ be two such Lyndon words with $|z_1|\le|z_2|$.
Since $|z_i|\le\ell+1$, the label $e_{z_i}$ begins with $z_i$ for $i=1,2$.
If $|z_1|=|z_2|$, then $e_{z_1}$ and $e_{z_2}$ already differ in their first $|z_1|$ letters.
If $|z_1|<|z_2|$ and $e_{z_1}=e_{z_2}$, then $z_2$ is a prefix of the infinite repetition of $z_1$ and has period $|z_1|<|z_2|$.
A Lyndon word is lexicographically smaller than each of its nonempty proper suffixes~\cite{lothaire}.
A shorter period would give it a nonempty proper suffix that is also a prefix and hence lexicographically smaller than the word, a contradiction (cf.\ Lemma~3 of~\cite{tomita-i}).

Consider a linear relation $\sum_z\alpha_z\,\mu(C_z(\ell))=0$, and suppose that some $\alpha_z$ is nonzero.
Among the Lyndon words $z$ with $\alpha_z\ne0$, choose the one with the smallest $e_z$.
Every other circuit with a nonzero coefficient has a larger least label, so it does not traverse the arc labelled $e_z$.
The coordinate of the left-hand side at this arc is therefore $\alpha_z$ times a positive integer, a contradiction.
Hence every $\alpha_z$ is zero, so the circuits are linearly independent and, in particular, pairwise distinct.
\end{proof}

\begin{corollary}[Number of circuits at a fixed order]\label{cor:graph}
For $1\le\ell\le m-2$,
\begin{equation}\label{eq:graph-bound}
 b_\ell\le d_\ell+1,\qquad d_\ell=0\ \Longrightarrow\ b_\ell=0 .
\end{equation}
\end{corollary}

\begin{proof}
For each root $z$ counted at order $\ell$, the graph $\Gamma_W(\ell)$
contains $C_z(\ell)$ and $p_z\le\ell$. Applying
Lemma~\ref{lem:independence} with $Y=W$ gives $b_\ell$ independent
circuits in this graph.
The bound on independent circuits established in
Section~\ref{sec:rauzy} therefore gives $b_\ell\le d_\ell+1$.

If $d_\ell=0$, Lemma~\ref{lem:cycle} gives a single elementary
circuit of length $m$. Every circuit counted by $b_\ell$ would
be elementary of length $p_z<m$, so $b_\ell=0$.
\end{proof}

\begin{proposition}[Squares of length less than $2m$]\label{prop:common}
The squares $xx$ in $[w]$ with $|x|<m$ satisfy
\begin{equation}\label{eq:common}
 \bigl|\{xx\in\Sq([w]): |x|<m\}\bigr|=\sum_z s_z\le 2(m-\sigma).
\end{equation}
\end{proposition}

\begin{proof}
Corollary~\ref{cor:graph} gives $b_\ell\le2d_\ell$ at every order in the
sum, since $b_\ell=0$ for $d_\ell=0$ and $d_\ell+1\le2d_\ell$ for
$d_\ell\ge1$. By
Lemma~\ref{lem:root-count},~\eqref{eq:double-count},
and~\eqref{eq:telescoping},
\[
 \sum_z s_z\le\sum_z K_z=\sum_{\ell=1}^{m-2}b_\ell
 \le 2\sum_{\ell=1}^{m-2}d_\ell=2(m-\sigma).
\]
\end{proof}

\section{Proof of the main theorem}\label{sec:proof}
We now prove Theorem~\ref{thm:main}.
For nonprimitive words, the decomposition~\eqref{eq:root-data}
and Proposition~\ref{prop:common} give the required bound.
For primitive words, we refine the circuit count by assigning
weights to the circuits to obtain the coefficient $3/2$.

\subsection{Nonprimitive words}\label{sec:nonprimitive}
Suppose $k\ge2$. Applying Proposition~\ref{prop:common} to the
second term of~\eqref{eq:root-data} gives
\[
 S([w])\le m\lfloor k/2\rfloor+2(m-\sigma)
       \le \frac{n}{2}+2m-2\sigma
       \le \frac{3n}{2}-2\sigma
       \le \frac32(n-\sigma).
\]
The second and third inequalities use $m\lfloor k/2\rfloor\le n/2$
and $2m\le n$, respectively, since $n=km$ and $k\ge2$. If $m=1$, then $\sigma=1$, $n=k\ge2$, and
$S([w])=\lfloor n/2\rfloor\le 3n/2-2$. Thus the stronger bound
$3n/2-2\sigma$ holds for every nonprimitive word.

\subsection{The exceptional case for primitive words}\label{sec:exceptional}

Assume henceforth that $w$ is primitive, so $k=1$ and $m=n$.
By ~\eqref{eq:root-data}, Lemma~\ref{lem:root-count}, and~\eqref{eq:double-count},
\[
 S([w])=\sum_z s_z\le\sum_{\ell=1}^{n-2}b_\ell.
\]
Since $\sum_{\ell=1}^{n-2}d_\ell=n-\sigma$
by~\eqref{eq:telescoping}, the desired bound would follow if
\[
 b_\ell\le\frac32d_\ell
 \qquad\text{for every }1\le\ell\le n-2.
\]

Corollary~\ref{cor:graph} gives $b_\ell=0$ when $d_\ell=0$,
and $b_\ell\le d_\ell+1\le3d_\ell/2$ when $d_\ell\ge2$.
Thus the desired inequality can fail only when
$d_\ell=1$ and $b_\ell=2$.
We call such an order \emph{exceptional}.
At an exceptional order, the two counted circuits contribute $2$,
exceeding the desired bound $3/2$ by $1/2$.

We start by giving every counted circuit weight $1$.
At each exceptional order, we will reduce the weight of one of
its two circuits to $1/2$.
This makes the total weight at every order at most $3d_\ell/2$.
We must ensure that, after all reductions, the circuits with each
root $z$ still have total weight at least $s_z$.
The next lemma shows that, at each exceptional order, one of the
two roots $z$ is also counted at the next $p_z$ orders.
We use this fact in Section~\ref{sec:weighted} to prove that
the required total weight is preserved for every root.

For $z\in\Zset$, take a square in $[w]$ with root $z$. Its length is
between $2p_z$ and $n$. If $2p_z=n$, the square has length $n$, so
its bi-infinite repetition is a shift of $W$ with period $n/2$. This
contradicts the least period $n$ of $W$, proving $2p_z<n$.
We use this strict inequality in the proof of Lemma~\ref{lem:extension}.

\begin{lemma}[Persistence of one root]\label{lem:extension}
Let $w$ be primitive and let $1\le\ell\le n-2$.
If $d_\ell=1$ and $b_\ell=2$, then at least one of the two roots
$z\in\Zset$ counted at order $\ell$ satisfies
\[
M_z\ge\ell+p_z.
\]
\end{lemma}

\begin{proof}
We first show that the walk read from $W$ traverses one of the
two circuits twice consecutively.
By Lemma~\ref{lem:connect}, $\Gamma_W(\ell)$ is strongly connected
and every vertex has positive outdegree. Since $d_\ell=1$, equation~\eqref{eq:outdegrees} implies that exactly one vertex $v_0$ has outdegree two and every other vertex has outdegree one.
After either arc leaving $v_0$ is chosen, the subsequent arcs are
forced until the next visit to $v_0$. The walk cannot revisit any
other vertex first, since that would give a circuit with no exit,
contradicting strong connectivity. The two choices therefore give
elementary circuits. An elementary circuit avoiding $v_0$ would also
have no exit, so these are the only two. For each counted root $z$, the circuit $C_z(\ell)$ is elementary since
$p_z\le\ell$. The two counted circuits are distinct by
Lemma~\ref{lem:independence}, so they must be the two just described. Denote their roots by
$z_1,z_2$, and write $A=C_{z_1}(\ell)$ and $B=C_{z_2}(\ell)$. Both pass
through $v_0$ and may share further vertices or arcs.

The walk read from $W$ visits $v_0$ infinitely often in both
directions. Between successive visits it traverses either $A$ or
$B$. If neither circuit is traversed twice consecutively, the walk
alternates between them. The arc labels, and therefore their first letters,
then repeat every $p_{z_1}+p_{z_2}$ steps. This would give $W$ a period
$p_{z_1}+p_{z_2}<n$, since $2p_{z_1}<n$ and $2p_{z_2}<n$, contrary to its
least period $n$.

Choose $z\in\{z_1,z_2\}$ so that $C_z(\ell)$ is traversed twice consecutively,
and write $p=p_z$. We show that $C_z(\ell+p)$ occurs in
$\Gamma_W(\ell+p)$.
If the first traversal starts at position $i_0$,
the two traversals read the factor
$U=W[i_0..i_0+\ell+2p-1]$ of length $\ell+2p$. Here $U$ is a
factor of the bi-infinite word $W$, so its length may exceed $n$.
The vertices before, between, and after the traversals coincide, so
\begin{equation}\label{eq:repeated-factors}
 U[1..\ell]=U[p+1..p+\ell]=U[2p+1..2p+\ell].
\end{equation}
These equalities give $U[i]=U[i+p]$ for
$1\le i\le\ell$ and for $p+1\le i\le p+\ell$.
Since $\ell\ge p$, these ranges cover all integers from $1$ to
$\ell+p=|U|-p$. Thus $p$ is a period of $U$.
The length-$\ell$ prefix of $U$ labels a vertex of $C_z(\ell)$,
so the first $p$ letters of $U$ form a rotation of $z$. Consequently, $U$ is a factor
of $z^{\mathbb Z}$.

For each $i=1,\ldots,p$, take the length-$(\ell+p+1)$ factor
of $U$ starting at $i$. The last of these factors ends at
$p+(\ell+p+1)-1=\ell+2p=|U|$.
Since $z^{\mathbb Z}$ has period $p$, each of its length-$(\ell+p+1)$
factors occurs in $U$ starting at one of the positions $1,\ldots,p$.
Since $U$ occurs in $W$, we have
$\Fac_{z^{\mathbb Z}}(\ell+p+1)\subseteq\Fac_W(\ell+p+1)$.
Thus every arc of $C_z(\ell+p)$ occurs in $\Gamma_W(\ell+p)$,
which gives $M_z\ge\ell+p_z$.
\end{proof}

\paragraph*{Example.}
Let $w=\word{abaabaabab}$, which is primitive, and let $\ell=3$. 
At order~3 the graph has four vertices and five arcs, so
$d_3=1$. Its two elementary circuits have roots $\word{aab}$
and $\word{ab}$, and lengths $3$ and $2$, respectively
(Figure~\ref{fig:example}). Both roots occur among the squares in
$[w]$. $(\word{aba})^2$ starts at position~1 and
$(\word{ab})^2$ starts at position~7. Both circuits are therefore
counted at order~3, so $b_3=2$. Denote the two circuits by $A,B$,
respectively.

\begin{figure}[htbp]
\centering
\begin{minipage}[c]{.49\linewidth}
\centering
\begin{tikzpicture}[x=.9cm,y=.9cm,
 vertex/.append style={font=\ttfamily\small,minimum width=11mm,minimum height=7mm},
 elabel/.append style={font=\ttfamily\footnotesize}]
 \node[vertex] (v) at (0,0) {aba};
 \node[vertex] (x) at (-2,1.1) {baa};
 \node[vertex] (y) at (-2,-1.1) {aab};
 \node[vertex] (z) at (2.1,0) {bab};
 \draw[edgea] (v)--node[elabel,above right]{abaa}(x);
 \draw[edgea] (x)--node[elabel,left]{baab}(y);
 \draw[edgea] (y)--node[elabel,below right]{aaba}(v);
 \draw[edgeb] (v) to[bend left=38] node[elabel,above]{abab}(z);
 \draw[edgeb] (z) to[bend left=38] node[elabel,below]{baba}(v);
\end{tikzpicture}\\[3pt]
$\Gamma_W(3)$
\end{minipage}\hfill
\begin{minipage}[c]{.49\linewidth}
\centering
\begin{tikzpicture}[x=.95cm,y=.9cm,
 vertex/.append style={font=\ttfamily\small,minimum height=7mm},
 elabel/.append style={font=\ttfamily\footnotesize}]
 \node[vertex] (a) at (0,1.0) {abaaba};
 \node[vertex] (b) at (1.9,-1.0) {baabaa};
 \node[vertex] (c) at (-1.9,-1.0) {aabaab};
 \draw[edgea] (a)--node[elabel,above right]{abaabaa}(b);
 \draw[edgea] (b)--node[elabel,below]{baabaab}(c);
 \draw[edgea] (c)--node[elabel,above left]{aabaaba}(a);
\end{tikzpicture}\\[3pt]
$C_{\word{aab}}(6)$
\end{minipage}
\caption{In $\Gamma_W(3)$ (left), the solid arcs form
$A=C_{\word{aab}}(3)$ and the dashed arcs form $B=C_{\word{ab}}(3)$.
The right panel shows $C_{\word{aab}}(6)$ in $\Gamma_W(6)$.}
\label{fig:example}
\end{figure}

Starting at position~1, the walk visits $\word{aba}$ at positions
$1,4,7,9,11$, tracing $A,A,B,B$ between successive visits.
Here $A$ has length $p=3$, and its two consecutive traversals
give the factor $U=\word{abaabaaba}$.
Its three factors of length~7 form $\Fac_{(\word{aab})^{\mathbb Z}}(7)$
(Figure~\ref{fig:overlapping-factors}), so $C_{\word{aab}}(6)$
is present in $\Gamma_W(6)$. The two traversals of $B$ read
$W[7..13]=\word{abababa}$. Its two length-6 factors likewise give
$C_{\word{ab}}(5)$. Either root therefore satisfies
Lemma~\ref{lem:extension} at order~3. We continue the count for this
word in Section~\ref{sec:weighted}.

\begin{figure}[htbp]
\centering
\begin{tikzpicture}[x=.67cm,y=.8cm]
\foreach \c [count=\j from 1] in {a,b,a,a,b,a,a,b,a}{
 \node[font=\scriptsize] at (\j,3.2) {\j};
 \node[font=\ttfamily] at (\j,2.7) {\c};
}
\node at (-.15,2.7) {$U$};
\foreach \s in {1,2,3}{
 \pgfmathsetmacro{\y}{2.1-.72*\s}
 \pgfmathtruncatemacro{\e}{\s+6}
 \draw[fill=black!5,rounded corners=2pt] (\s-.45,\y-.28) rectangle (\e+.45,\y+.28);
 \node at (-.15,\y) {$i=\s$};
 \foreach \c [count=\j from 1] in {a,b,a,a,b,a,a,b,a}{
  \ifnum\j<\s\else\ifnum\j>\e\else
   \node[font=\ttfamily] at (\j,\y) {\c};
  \fi\fi
 }
}
\end{tikzpicture}
\caption{The length-$7$ factors of $U=\word{abaabaaba}$ starting at positions $1,2,3$. Here $\ell=p=3$, and the last factor ends at position $\ell+2p=9$.}\label{fig:overlapping-factors}
\end{figure}

\FloatBarrier
\subsection{Weighted counting for primitive words}\label{sec:weighted}

Assign weight $1$ to each $C_z(\ell)$ with
$z\in\Zset$ and $p_z\le\ell\le M_z$, and let
\[
 \mathcal B=\{\ell:1\le\ell\le n-2,\ d_\ell=1,\ b_\ell=2\}
\]
be the set of exceptional orders.
For each $\ell\in\mathcal B$, Lemma~\ref{lem:extension} gives a root
$z$ counted at order $\ell$ with $M_z\ge\ell+p_z$. Choose one
such root and denote it by $f(\ell)$. Reduce the weight of the
circuit with this \emph{chosen root} at order $\ell$ to $1/2$. If both roots qualify, either choice is allowed.
Let $T_z=|\{\ell\in\mathcal B:f(\ell)=z\}|$. Each order in
$\mathcal B$ contributes to exactly one $T_z$, so
\begin{equation}\label{eq:selections}
 \sum_z T_z=|\mathcal B|.
\end{equation}
The total weight of the circuits of root $z$
is then $K_z-T_z/2$.

\begin{lemma}[Remaining weight for one root]\label{lem:reduction}
For every $z\in\Zset$,
\begin{equation}\label{eq:reduction}
 K_z-\frac{T_z}{2}\ge s_z.
\end{equation}
\end{lemma}

\begin{proof}
If $f(\ell)=z$, the circuit $C_z(\ell)$ is present and
$M_z\ge\ell+p_z$. Hence
\begin{equation}\label{eq:eligible}
 p_z\le\ell\le M_z-p_z.
\end{equation}
This interval contains $\max\{0,M_z-2p_z+1\}=\max\{0,K_z-p_z\}$
integers. Since $T_z$ counts the distinct values of $\ell$ for
which $z$ is chosen, we have
\[
 T_z\le\max\{0,K_z-p_z\}.
\]
The two inequalities in Lemma~\ref{lem:root-count} give
$2(K_z-s_z)\ge0$ and $2(K_z-s_z)\ge K_z-p_z$, respectively.
Thus $T_z\le2(K_z-s_z)$, which gives $K_z-T_z/2\ge s_z$.
\end{proof}

At each order $\ell\in\mathcal B$, the two circuits now have total
weight $3/2=3d_\ell/2$. At each remaining order $1\le\ell\le n-2$, the total weight is
$b_\ell\le3d_\ell/2$ by the discussion preceding
Lemma~\ref{lem:extension}. Summing
over all orders gives
\begin{equation}\label{eq:total-weight}
 \sum_{\ell=1}^{n-2}b_\ell-\frac{|\mathcal B|}{2}
 \le\frac32\sum_{\ell=1}^{n-2}d_\ell.
\end{equation}

\paragraph*{Example continued.}
For $w=\word{abaabaabab}$ from Section~\ref{sec:exceptional},
the roots $\word{a},\word{ab},\word{aab}$ account for $1,2,3$
squares, respectively. Their counted orders are $\{1\}$,
$\{2,\ldots,5\}$, and $\{3,\ldots,6\}$. 
For $\ell=1,\ldots,8$, direct enumeration gives
$(d_\ell)=(1,1,1,2,2,1,0,0)$, so order~3 alone is exceptional.
We may choose $f(3)=\word{aab}$ because $M_{\word{aab}}=6=3+3$.
The two circuits at order~3 now have total weight $3/2$.
The totals for the three roots are $1,4,7/2$, giving
$S([w])=6\le17/2\le3(10-2)/2=12$.

\begin{proof}[Proof of Theorem~\ref{thm:main}]
For primitive $w$ with $n\ge2$, Lemma~\ref{lem:reduction},
equations~\eqref{eq:double-count} and~\eqref{eq:selections}, and
inequality~\eqref{eq:total-weight} give
\begin{align}
 S([w])=\sum_zs_z
 &\le\sum_z\left(K_z-\frac{T_z}{2}\right)\notag\\
 &=\sum_{\ell=1}^{n-2}b_\ell-\frac{|\mathcal B|}{2}\notag\\
 &\le\frac32\sum_{\ell=1}^{n-2}d_\ell
 =\frac32(n-\sigma),\label{eq:final}
\end{align}
where the last equality is~\eqref{eq:telescoping} with $m=n$.
Together with Section~\ref{sec:nonprimitive} and the unary case, this gives
$S([w])\le 3(n-\sigma)/2$ for every nonempty word. Taking the floor
proves the theorem.
\end{proof}

\section{Concluding remarks}\label{sec:discussion}
We proved that a circular word of length $n$ containing $\sigma$
distinct letters has at most $\lfloor3(n-\sigma)/2\rfloor$
distinct squares.
The construction of Charalampopoulos et al.~\cite{circular} gives circular words of length~$n$ with $3n/2-O(\sqrt n)$ distinct squares for infinitely many~$n$.
Does every circular word of length~$n$ contain at most $3n/2-\Omega(\sqrt n)$ distinct squares?
The corresponding question for linear words is also open: Fraenkel and Simpson~\cite{fraenkel-simpson} constructed words of length~$n$ with $n-O(\sqrt n)$ distinct squares for infinitely many~$n$, whereas the best known upper bound is $n-\Theta(\log n)$~\cite{brlek-li-log,tomita-i}.

For fixed $j\ge3$, consider the maximum number of distinct $j$-th powers in a circular word of length~$n$, counting only powers of length at most~$n$.
Li, Pachocki, and Radoszewski~\cite[Theorem~13]{li-pachocki-radoszewski} proved that a linear word of length~$n$ containing $\sigma$ distinct letters has at most $(n-\sigma)/(j-1)$ distinct $j$-th powers.
Their construction shows that the leading coefficient $1/(j-1)$ is optimal for linear words.
What is the optimal leading coefficient for $j$-th powers in circular words?

\section*{Use of generative AI}
The author used ChatGPT and Codex (OpenAI) to explore proof strategies
and look for gaps and counterexamples.
These tools assisted with developing and checking the arguments
in Section~\ref{sec:nonprimitive} and the proof of
Lemma~\ref{lem:root-count}.
These tools, together with Claude (Anthropic), also assisted
with the exposition, figures, and \LaTeX{} source.
The author takes responsibility for all arguments, computations,
references, and the final manuscript.

\bibliography{references}
\end{document}